\documentclass[12pt,leqno,oneside]{amsart}
\usepackage[T1]{fontenc}
\usepackage{lmodern}
\usepackage{mathrsfs,dsfont}
\usepackage{amsmath,amstext,amsthm,amssymb,mathtools,bbm,faktor}
\usepackage{typearea}
\usepackage{hyperref}
\usepackage{color}
\usepackage{nccmath}
\usepackage{float}
\usepackage{tikz}
\usetikzlibrary{patterns}
\usepackage{graphicx}
\usepackage{hyperref}
\usepackage{cleveref}

\DeclareRobustCommand{\SkipTocEntry}[5]{}

\usepackage{enumitem}

\usepackage[width=6.4in,height=8.5in]
{geometry}

\newtheorem{thm}{Theorem}[section]
\newtheorem{prop}[thm]{Proposition}
\newtheorem{lem}[thm]{Lemma}
\newtheorem{cor}[thm]{Corollary}

\newtheorem*{fact}{Fact}

\theoremstyle{definition}

\newtheorem{defn}[thm]{Definition}
\newtheorem{example}[thm]{Example}
\newtheorem{rem}[thm]{Remark}

\numberwithin{equation}{section}
\newcommand{\DD}{\mathbb D}
\newcommand{\N}{\mathbb{N}}
\newcommand{\TT}{\mathbb T}
\newcommand{\CC}{\mathbb C}
\newcommand{\veps}{\varepsilon}

\begin{document}
	
	\title[Approximation Numbers]{Approximation Numbers of Differences of Composition Operators}
	
	\author{Fr\'{e}d\'{e}ric Bayart, Clifford Gilmore, Sibel \d{S}ahin}
	
	\date{\today}
	
	\thanks{F.~Bayart and C.~Gilmore were partially supported by the grant ANR-24-CE40-0892-01 of the French National Research
		Agency ANR. C.~Gilmore was also supported by the European Union’s Horizon Europe research and innovation programme under the Marie~Skłodowska-Curie grant agreement No.\ 101066064. Sibel \d{S}ahin was supported by The Scientific and Technological Research Council of T\"{u}rkiye (TUBITAK)-2219 International Postdoctoral Research Fellowship Program (Project no: 1059B192301690).}

	\address{Laboratoire de Math\'{e}matiques Blaise Pascal UMR 6620 CNRS\\ Universit\'{e} Clermont Auvergne, Campus universitaire des C\'{e}zeaux\\ 3 place Vasarely, 63178 Aubi\`{e}re Cedex, France.}
	\email{frederic.bayart@uca.fr}
	
	\address{Laboratoire de Math\'{e}matiques Blaise Pascal UMR 6620 CNRS\\ Universit\'{e} Clermont Auvergne, Campus universitaire des C\'{e}zeaux\\ 3 place Vasarely, 63178 Aubi\`{e}re Cedex, France.}
	\email{clifford.gilmore@uca.fr}

	\address{Department of Mathematics, Mimar Sinan Fine Arts University, Istanbul, T\"{u}rkiye}
	
	\email{sibel.sahin@msgsu.edu.tr}
	
	
	
	\begin{abstract}
	In this study we consider the approximation numbers of differences of composition operators acting on the Hardy-Hilbert space $H^2(\mathbb{D})$. We obtain both upper and lower bounds for these approximation numbers and by applying these general results to composition operators with specific types of symbols, we demonstrate the effect of boundary behaviour over the approximation numbers. Moreover, we use these one-dimensional methods and examples to understand the approximation numbers of differences of composition operators acting on the space $H^2(\mathbb{D}^2)$  of the bidisc. 
	\end{abstract}

	\maketitle

	\tableofcontents

	\section{Introduction}
	If $\varphi$ is an analytic self-map of the unit disc $\DD,$ the composition operator $C_\varphi$ is defined on the Hardy space $H^2(\DD)$ by $C_\varphi f=f\circ\varphi.$ An important problem in the study of composition operators is to characterise when the difference of two composition operators is compact. This is linked to the study of the topology of the set $\mathcal{C}(H^2)$ of (bounded) composition operators that act on the Hardy space $H^2(\DD)$. 
	For a pair $\varphi$ and $\psi$ of analytic self-maps of $\DD$, Shapiro and Sundberg~\cite{SS90} conjectured that $C_\varphi$ and $C_\psi$ belong to the same component of $\mathcal C(H^2)$ if and only if $C_\varphi-C_\psi$ is a compact operator on $H^2(\DD)$.
	This conjecture was independently disproved by Moorhouse and Toews~\cite{MT} and Bourdon~\cite{BOU}. Nevertheless, a significant literature has subsequently emerged that is devoted to understanding when the difference of two composition operators is compact (cf.\ for instance \cite{BWY, CKP,GK,Moo05,Sau}).
	
	In this article, we go a step further by considering the degree of compactness of differences  $C_\varphi-C_\psi$ via approximation numbers.
	
	\begin{defn}
		The \emph{approximation numbers} of an operator $T \colon H_1 \to H_2$ acting between two Hilbert spaces $H_1$ and $H_2$ are defined for $n \in \N$ as
		$$
		a_n(T) = \inf_{\mathrm{rank} R<n} \|T-R\|.
		$$
	\end{defn}   
	This definition thus gives the distance, in operator norm, of $T$ to the operators of rank less than $n$, and it is well known that $T$ is compact if and only if $\lim_{n \to \infty} a_n(T) = 0$.
	
	Approximation numbers of composition and weighted composition operators have been extensively studied for symbols from various function spaces (cf.~\cite{BLQRP,LeLQRP,LQRP12,LQRP13,LQRP14,LQRP15,LQRP17,LQRP19}). We will combine the techniques from the aforementioned articles with that of \cite{Sau} in order to compute estimates of the decay of $a_n(C_\varphi-C_\psi)$. 
	To give a flavour of our results, we mention two examples where we will obtain precise estimates.

	The first example considers two non-compact composition operators that give rise to a compact difference. More specifically, let $\varphi(z)=\frac{1+z}2$	and $\psi(z)=\varphi(z)+c(z-1)^\alpha$, for $\alpha>2$ and $c\in(0,1/128)$, so that $\psi$ is a self-map of $\DD$ (cf.~\cite[p.\ 337]{CMc}). 
	It was shown in \cite{MT} that $C_\varphi-C_\psi$ is compact if and only if $\alpha>2$. It is also not difficult to study when this operator is Hilbert-Schmidt. 
	Indeed, from \cite{HJM} we know that the Hilbert-Schmidt norm of $C_\varphi-C_\psi$ (for general holomorphic self-maps $\varphi,\psi \colon \DD\to\DD$) is equal to 
	$$\|C_\varphi-C_\psi\|_{HS}=\frac 1{2\pi}\int_{-\pi}^{\pi} 
	\frac{1-|\varphi(e^{i\theta})|^2|\psi(e^{i\theta})|^2}
	{(1-|\varphi(e^{i\theta})|^2)(1-|\psi(e^{i\theta})|^2)}
	\left|\frac{\varphi(e^{i\theta})-\psi(e^{i\theta})}{1-\overline{\varphi(e^{i\theta})}\psi(e^{i\theta})}\right|^2 d\theta.$$
	In our particular case, one can observe that
	$$\frac{1-|\varphi(e^{i\theta})|^2|\psi(e^{i\theta})|^2}
	{(1-|\varphi(e^{i\theta})|^2)(1-|\psi(e^{i\theta})|^2)}  \left|\frac{\varphi(e^{i\theta})-\psi(e^{i\theta})}{1-\overline{\varphi(e^{i\theta})}\psi(e^{i\theta})}\right|^2\sim_0 \theta^{2\alpha-6}$$
	so that $C_\varphi-C_\psi$ is Hilbert-Schmidt if and only if $6-2\alpha<1$, i.e.\ $\alpha>5/2.$
	
	We will be able to give rather precise decay estimates of $a_n(C_\varphi-C_\psi)$.
	\begin{example}\label{ex:differencesmooth}
		Let $\varphi(z)=\frac{1+z}2$ and $\psi(z)=\varphi(z)+c(z-1)^\alpha$
		for $\alpha>2$ and $c\in(0,1/128).$ 
		Then we get that
		$$\frac 1{(\log n)n^{\alpha-2}} \lesssim a_n(C_\varphi-C_\psi)\lesssim \left(\frac {\log(n)}{n}\right)^{\alpha-2}.$$
	\end{example}
	
	As standard, the notation $f(x)\lesssim g(x)$ for $f,g \colon X\to\mathbb R$ means that there exists $C>0$
	such that, for all $x\in X,$ $f(x)\leq C g(x).$
	Observe that the exponent $\alpha-2$ is in line with the condition that characterises when 
	$C_\varphi-C_\psi$ is Hilbert-Schmidt, since $\sum_{n\geq 1}(a_n(C_\varphi-C_\phi))^2$ is convergent if and only if $\alpha>5/2.$
	
	\medskip
	
	Our second example deals with the difference of two composition operators such that the range of their symbol has a corner. An important class of symbols $\varphi$ such that $\varphi(\DD)$ touches the boundary and has a corner are the lens maps, and the approximation numbers of composition operators induced by lens maps have been explored in detail in \cite{LLQRP,LQRP12,LQRP19,LeLQRP2}. For simplicity, we will work with the variant 
	$$\varphi(z)=\frac1{1+(1-z)^{1/2}}.$$
	It has been shown in \cite{QS} that there exist $a,b>0$ such that 
	$$\exp(-a\sqrt n)\lesssim a_n(C_\varphi)\lesssim \exp(-b\sqrt n).$$
	We now perturb $\varphi$ by considering the function $\psi(z)=\varphi(z)+c \chi(z)$, where 
	$$\chi(z)=\exp\left(-\frac 1{(1-z)^{1/2}}\right).$$
	We will show later that for $c>0$ sufficiently small, $\psi$ remains a self-map of $\DD.$ It turns out that the difference $C_\varphi-C_\psi$ has much smaller approximation numbers than either $C_\varphi$ or $C_\psi$: there exist $a,b>0$ such that 
	$$\exp\left(-a\frac{n}{\log n}\right)\lesssim a_n(C_\varphi-C_\psi)\lesssim \exp\left(-b\frac{n}{\log n}\right).$$
	
	\addtocontents{toc}{\SkipTocEntry}
	\subsection*{Organisation of the article}
	In Section 2, we introduce the basic tools like Carleson measures, interpolation sequences and pseudohyperbolic distance that we will use repeatedly throughout the article. In Section 3, we first obtain theoretical lower bounds for the approximation numbers of differences of composition operators and then we apply these results to a large class of examples. Section 4 is devoted to upper bounds of approximation numbers, starting again from a general result and then going into examples. Section 5 is devoted to the weighted composition operators and we improve some results of \cite{LeLQRP}. In the last section, we show how these one-dimensional techniques and examples can be used for calculating the approximation numbers of differences in the bidisc case for some specific classes of symbols.

	\section{Preliminaries}
	
	\subsection{Pseudohyperbolic and hyperbolic distance}
	The \emph{pseudohyperbolic distance} between two points $z$ and $w$ of $\DD$ will be denoted by $\rho(z,w)$ and it is defined as
	$$\rho(z,w)=\frac{|z-w|}{|1-\bar z w|}.$$
	We will denote by $d(z,w)$ the \emph{hyperbolic distance} between two points $z, w \in \DD$, which is normalised to give
	$$\rho(z,w)=\frac{1-e^{-d(z,w)}}{1+e^{-d(z,w)}}.$$
	In particular, the hyperbolic length of a curve $\Gamma$ contained in $\DD$ is given by
	$$\ell_P(\Gamma)=2\int_{\Gamma}\frac{|dz|}{1-|z|^2}.$$
	
	\subsection{Carleson measures and interpolating sequences}
	
	Our results will strongly depend on classical (and deep) results on Carleson measures and interpolating sequences. Let us first recall the relevant definitions.
	
	\begin{defn}
		A non-negative Borel measure $\mu$ on the closed unit disc $\overline \DD$ is called a \emph{Carleson measure} if there exists a constant $C>0$ such that, for all Carleson squares 
		$$Q(\theta_0,\delta) \coloneqq \{z=re^{i\theta}\in\overline{\DD} : \ r\geq 1-\delta,\ |\theta-\theta_0|\leq\delta\},$$
		the following inequality holds:
		$$\mu\big(Q(\theta_0,\delta)\big)\leq C\delta.$$
		The smallest such constant $C$ is called the \emph{Carleson norm} 
		of $\mu$ and will be denoted by $\|\mu\|_{\mathcal C}$. 
	\end{defn}
	
	Carleson measures are linked to embedding operators from both the holomorphic and harmonic Hardy spaces into $L^2(\overline{\DD},\mu)$. We denote by $Pf$ 
	the Poisson transform of a function $f\in L^1(\TT)$ and defined by 
	$$Pf(z)=\frac 1{2\pi}\int_0^{2\pi}\frac{1-|z|^2}{|1-\bar z e^{i\theta}|^2}f(e^{i\theta})d\theta$$
	for every $z\in\DD.$
	
	\begin{thm}\label{thm:embedding}
		Let $\mu$ be a non-negative Borel measure  on $\overline{\DD}$. The following assertions are equivalent.
		\begin{enumerate}[label=(\roman*)]
			\item $\mu$ is a Carleson measure.
			\item The map 
			\begin{align*}
				J_\mu^{\mathrm hol} \colon H^2(\DD)&\to L^2(\overline{\DD},\mu), \\
				f&\mapsto f
			\end{align*}
			is well defined and bounded. 
			\item The map
			\begin{align*}
				J_\mu^{\mathrm har} \colon L^2(\TT)&\to L^2(\overline{\DD},\mu), \\
				f&\mapsto Pf
			\end{align*}
			is well defined and bounded.
		\end{enumerate}
		Moreover, there exist absolute constants $a,b$ such that, if 
		the previous conditions hold, then 
		$$\|\mu\|_\mathcal C\leq a\|J_\mu^{\mathrm hol}\| \leq a\|J_\mu^{\mathrm har}\|\leq b\|\mu\|_{\mathcal C}.$$
	\end{thm}
	This result is a folklore theorem, although we did not find it written in this form in the litterature. That (iii)$\implies$(ii) is immediate, and that (i)$\iff$(ii)
	is \cite[Theorem 2.35]{CMc} (the equivalence of norms
	could be easily deduced from its proof and the proof of 
	\cite[Theorem 2.33]{CMc}). The proof of (i)$\implies$(iii),
	with $\|J_\mu^{\mathrm har}\|\leq c\|\mu\|_{\mathcal C},$
	is done in \cite[Theorem 5.6 of Chapter 1]{Gar} when $\mu$ is supported on $\DD$. 
	
	If $\mu$ is a Carleson 
	measure supported on $\overline{\DD}$, we may write it as $\mu=\mu_{\DD}+\mu_{\TT}$, where $\mu_{\DD}$
	is supported on $\DD$ and $\mu_{\TT}$ is supported on the unit circle $\TT$.
	Since, for all intervals $I(\theta_0,\delta)=\{z=e^{i\theta}:\ |\theta-\theta_0|<\delta\}$, it holds that
	$$\mu_{\TT}(I(\theta,\delta))\leq \mu(Q(\theta_0,\delta))\leq \|\mu\|_{\mathcal C}\delta,$$
	we get that $\mu_\TT$ is absolutely continuous with respect to the Lebesgue 
	measure $\sigma$ and we can write $\mu_{\TT}=Fd\sigma$ 
	with $\|F\|_\infty\leq \|\mu\|_{\mathcal C}.$
	The full proof of (i)$\implies$(iii) now follows easily.
	
	\begin{defn}
		A sequence $Z=(z_j)$ is a \emph{Carleson sequence} for $H^2(\mathbb{D})$ if the measure
		$$\nu_Z=\sum_j (1-|z_j|^2)\delta_{z_j}$$
		is a Carleson measure for $H^2(\DD)$. 
	\end{defn}
	
	The Carleson norm of $\nu_Z$ can be used to estimate the norm of linear combinations of reproducing kernels. Indeed, by \cite[Lemma 2.1]{QS}, if $Z$ is a Carleson sequence for $H^2(\DD)$,
	for any finite sequence of complex numbers $(b_j),$ 
	\begin{equation}\label{eq:carlesonineq}
		\left\|\sum_j b_j k_{z_j}\right\|^2\leq \|v_Z\|_{\mathcal C}\sum_j \dfrac{|b_j|^2}{1-|z_j|^2}
	\end{equation}
	where $k_w$ denotes the reproducing kernel at $w\in\DD$ defined by
	$$k_w(z)=\frac 1{1-\overline{w}z}.$$
	
	\begin{defn}
		A sequence $Z=(z_j)$ is an \emph{interpolating sequence} for $H^2(\mathbb{D})$ if the interpolation problem $f(z_j)=a_j$ has a solution $f\in H^2(\mathbb{D})$ whenever $\sum_{j}|a_j|^2(1-|z_j|^2)<\infty$. When this definition is satisfied, then there exists a constant $C$ with 
		\begin{equation}\label{eq:intconstant}
			\|f\|_{H^2}\leq C \left(\sum_{j}|a_j|^2(1-|z_j|^2)\right)^{1/2}
		\end{equation}
		and the smallest $C$ satisfying \eqref{eq:intconstant} is called the \emph{constant of interpolation} of $Z$, which we will denote by $M(Z)$.
	\end{defn}
	
	We will need concrete estimates for $M(Z)$. This has been done by Carleson and then by Shapiro and Shields who gave geometric characterizations of interpolating sequence. For a sequence $Z=(z_j),$ its \emph{uniform separation constant} is defined by 
	$$\delta(Z)=\inf_j \prod_{k\neq j}\rho(z_j,z_k).$$
	The Shapiro and Shields theorem gives that 
	\begin{equation}\label{EQ:ShapiroShields}
		M(Z)\leq \frac{\|\nu_Z\|_{\mathcal C}^{1/2}}{\delta(Z)}.
	\end{equation}
	
	While one can estimate the Carleson norm of $\nu_Z$ using its geometric definition, we can also use the following estimate, which can be found, for instance, in the proof of \cite[Theorem 1.1 of Chapter 7]{Gar}:
	\begin{equation}\label{eq:carlesonnorm}\|\nu_Z\|_{\mathcal C}\leq c\left(1+\log\left(\frac1{\delta(Z)}\right)\right)
	\end{equation}
	for some absolute constant $c.$
	
	\subsection{A self-map of $\DD$}    \label{sec:map}
	Let
	$$\varphi(z)=\frac1{1+(1-z)^{1/2}}$$
	and $\psi(z)=\varphi(z)+c \chi(z)$
	with 
	$$\chi(z)=\exp\left(-\frac 1{(1-z)^{1/2}}\right).$$
	We intend to prove that for $|c|$ small enough, $\psi$ is a self-map of $\DD$. It is well known that $\varphi(\DD)\subset\DD$ and that, for any $z\in\TT,$ $\varphi(z)\in\TT$ if and only if $z=1.$
	It is also easy to check that $\chi(\DD)\subset\frac 12\DD.$ 
	At the neighbourhood of $1,$
	\begin{align*}
		\psi(e^{it})&=\frac1{1+(-it+o(t))^{1/2}}+c\exp\left(-\frac1{(it+o(t))^{1/2}}\right)\\
		&=\frac1{1+e^{i3\pi/4}|t|^{1/2}+o(|t|^{1/2})}+o(|t|^{1/2})
	\end{align*}
	where the  constant involved in $o(|t|^{1/2})$ does not depend on $c$ provided $|c|\leq 1$. We deduce that $|\psi(e^{it})|^2=1-\sqrt 2|t|^{1/2}+o(|t|^{1/2}).$
	We therefore get that there exists $t_0>0$ such that, for all $|t|\leq t_0$ and all $|c|\leq 1,$ $|\psi(e^{it})|\leq 1.$ We may then adjust $|c|$ to be small enough so that this remains true for $t_0\leq|t|\leq\pi.$

	\section{Lower Bounds for the Approximation Numbers of Differences of Composition Operators}
	
	\subsection{A general result}
	In this subsection, we give a lower bound for the approximation number of the difference of two composition operators inspired by \cite[Theorem 3.1]{QS}.
	Let $\varphi$ and $\psi$ be holomorphic self-maps of the unit disc and let $Z$ be an interpolating sequence for $H^2(\mathbb{D})$. For the remainder of this section we set $W=\varphi(Z)\cup\psi(Z)$ and we assume that $\varphi(Z)\cap\psi(Z)=\varnothing.$
	
	\begin{lem}\label{lem:lowerinterpolation}
		If $W$ is an interpolating sequence for $H^2(\mathbb{D})$ 
		and if $(b_j)$ is a finite sequence of complex numbers, then
		$$
		\|\sum_{j}b_j(k_{\varphi(z_j)}-k_{\psi(z_j)})\|_{H^2}^{2}\geq M(W)^{-2}\sum_{j}|b_j|^2(\|k_{\varphi(z_j)}\|^2+\|k_{\psi(z_j)}\|^2).
		$$
	\end{lem}
	
	\begin{proof}
		We argue by duality. First note that
		\begin{align*}
					\|\sum_{j}b_j(k_{\varphi(z_j)}-k_{\psi(z_j)})\|_{H^2} &=\sup_{\|f\|_{2}=1}\sum_{j}\langle b_j(k_{\varphi(z_j)}-k_{\psi(z_j)}),f\rangle \\ &=\sup_{\|f\|_{2} = 1}\sum_{j} b_j(f(\varphi(z_j))-f(\psi(z_j))).
		\end{align*}
		Next take the solution $g\in H^2(\mathbb{D})$ of the interpolation problem 
		$$
		g(\varphi(z_j))=\overline{b_j}(1-|\varphi(z_j)|^2)^{-1} \quad \textnormal{ and } \quad g(\psi(z_j))=-\overline{b_j}(1-|\psi(z_j)|^2)^{-1}.
		$$
		For this solution $g$, by (\ref{eq:intconstant}) we have that
		$$
		\|g\|_2\leq M(W) \left( \sum_{j}|b_j|^2 \left((1-|\varphi(z_j)|^2)^{-1}+(1-|\psi(z_j)|^2)^{-1} \right) \right)^{1/2}.
		$$
		Then for $f=\dfrac{g}{\|g\|_2}$,
		\begin{align*}
			\|\sum_{j}b_j(k_{\varphi(z_j)}-k_{\psi(z_j)})\|_{H^2}&=\sup_{\|f\|_{2}=1}\sum_{j} b_j(f(\varphi(z_j))-f(\psi(z_j)))\\
			&\geq \dfrac{\sum_{j}|b_j|^2((1-|\varphi(z_j)|^2)^{-1}+(1-|\psi(z_j)|^2)^{-1})}{\|g\|_2}\\
			&\geq M(W)^{-1} \dfrac{\sum_{j}|b_j|^2((1-|\varphi(z_j)|^2)^{-1}+(1-|\psi(z_j)|^2)^{-1})}{\left(\sum_{j}|b_j|^2((1-|\varphi(z_j)|^2)^{-1}+(1-|\psi(z_j)|^2)^{-1})\right)^{1/2}}
		\end{align*}
		which gives us
		$$
		\|\sum_{j}b_j(k_{\varphi(z_j)}-k_{\psi(z_j)})\|_{H^2}^{2}\geq M(W)^{-2}\sum_{j}|b_j|^2(\|k_{\varphi(z_j)}\|^2+\|k_{\psi(z_j)}\|^2).
		$$
	\end{proof}

	We are now ready to state the main result of this subsection concerning the lower bounds for the approximation numbers of differences of composition operators.
	
	\begin{thm}\label{thm:lowerbound}
		Let $Z=\{z_1,\dots, z_n\}$ be a finite sequence of $n$ distinct points such that $\mathrm{card}(W) = \mathrm{card}(\varphi(Z)\cup\psi(Z))=2n$. Then
		$$
		a_n(C_\varphi-C_\psi)\geq M(W)^{-1}\|v_Z\|_{\mathcal{C}}^{-1/2}\left(\inf_{1\leq j\leq n}\left(\dfrac{1-|z_j|^2}{1-|\varphi(z_j)|^2}+\dfrac{1-|z_j|^2}{1-|\psi(z_j)|^2}\right)\right)^{1/2}.
		$$   
	\end{thm}
	
	\begin{proof}
		Let $E(Z)=\mathrm{span}\{k_{z_1},\dots,k_{z_n}\}$. Then by \cite[Lemma 2.3]{QS}
		$$
		a_n(C_{\varphi}-C_{\psi})\geq \inf\limits_{\substack{f\in E(Z)\\ \|f\|_2=1}}\|(C^{*}_{\varphi}-C^{*}_{\psi})(f)\|.
		$$
		We take $f\in E(Z)$, $f=\sum_{j}b_jk_{z_j}$ with $\|f\|_2=1$. By Lemma \ref{lem:lowerinterpolation} we have that
		\begin{align*}
			\|(C^{*}_{\varphi}-C^{*}_{\psi})(f)\|_{2}^{2}&=	\|\sum_{j}b_j(k_{\varphi(z_j)}-k_{\psi(z_j)})\|_{2}^{2} \\
			&\geq M(W)^{-2}\sum_{j}|b_j|^2\left(\dfrac{1}{1-|\varphi(z_j)|^2}+\dfrac{1}{1-|\psi(z_j)|^2}\right)\\
		\end{align*}
		Next set
		$$
		\mu_{\varphi}^{2}=\inf_{1\leq j\leq n}\dfrac{1-|z_j|^2}{1-|\varphi(z_j)|^2} \quad \text{ and } \quad \mu_{\psi}^{2}=\inf_{1\leq j\leq n}\dfrac{1-|z_j|^2}{1-|\psi(z_j)|^2}.
		$$
		Then 
		$$
		\|(C^{*}_{\varphi}-C^{*}_{\psi})(f)\|_{2}^{2}\geq M(W)^{-2}	\mu_{\varphi}^{2}\sum_{j}\dfrac{|b_j|^2}{1-|z_j|^2}+M(W)^{-2}	\mu_{\psi}^{2}\sum_{j}\dfrac{|b_j|^2}{1-|z_j|^2}
		$$
		and by (\ref{eq:carlesonineq}) we have that
		\begin{align*}
			\|(C^{*}_{\varphi}-C^{*}_{\psi})(f)\|_{2}^{2}&\geq M(W)^{-2}	\mu_{\varphi}^{2}\|v_Z\|_{\mathcal{C}}^{-1}\|f\|_{2}^{2}+M(W)^{-2}	\mu_{\psi}^{2}\|v_Z\|_{\mathcal{C}}^{-1}\|f\|_{2}^{2}
			\\
			&=M(W)^{-2}\|v_Z\|_{\mathcal{C}}^{-1}(\mu_{\varphi}^{2}+\mu_{\psi}^{2}).
		\end{align*}
		Hence,
		$$
		a_n(C_\varphi-C_\psi)\geq M(W)^{-1}\|v_Z\|_{\mathcal{C}}^{-1/2}\left(\inf_{1\leq j\leq n}\left(\dfrac{1-|z_j|^2}{1-|\varphi(z_j)|^2}+\dfrac{1-|z_j|^2}{1-|\psi(z_j)|^2}\right)\right)^{1/2}.
		$$
	\end{proof}
	
	To obtain concrete applications of Theorem \ref{thm:lowerbound}, we apply inequalities \eqref{EQ:ShapiroShields} and \eqref{eq:carlesonnorm} to get the following corollary. Recall that $\delta(W)$ denotes the uniform separation constant of $W$.
	\begin{cor}\label{cor:lowerbound}
		Let $Z=\{z_1,\dots, z_n\}$ be a finite sequence of $n$ distinct points such that $\mathrm{card}(W)=\mathrm{card}(\varphi(Z)\cup\psi(Z))=2n$. Then
		$$
		a_n(C_\varphi-C_\psi)\gtrsim \frac{\delta(W)\left(\inf_{1\leq j\leq n}\left(\dfrac{1-|z_j|^2}{1-|\varphi(z_j)|^2}+\dfrac{1-|z_j|^2}{1-|\psi(z_j)|^2}\right)\right)^{1/2}}{\left(1+\log\left(\frac 1{\delta(W)}\right)\right)^{1/2}\left(1+\log\left(\frac 1{\delta(Z)}\right)\right)^{1/2}
		}.
		$$   
	\end{cor}
	
	\subsection{Application to the difference of two non-compact composition operators}
	
	The main obstacle to applying the results of the previous subsection is to identify a suitable sequence $(z_j)$. Our first illustration shows that we may choose it for the difference of two non-compact composition operators, where one symbol is a small perturbation of the other. We first need the following definition.
	
	\begin{defn}
		We say that $\varphi \colon \overline{\DD}\to\DD$ belongs to $\mathcal C^2(\{1\})$
		if there exist $a,b\in\CC$ such that, for all $z\in \overline{\DD},$
		$$\varphi(z)=\varphi(1)+a(z-1)+b(z-1)^2+o((z-1)^2).$$
	\end{defn}
	The complex number $a$ coincides with the radial derivative $\varphi'(1)$ and belongs to $(0,1]$, and we set $\varphi''(1)=2b.$
	
	
	\begin{thm}\label{thm:exlowerbound}
		Let $\varphi,\psi \colon \overline{\DD}\to\overline{\DD}$ be holomorphic in $\DD$, non constant,  and let $C_1>0$, $\alpha>2.$ Assume that $\varphi$ belongs to $\mathcal C^2(\{1\})$, that $\varphi(1)=1$ and that $\psi(z)=\varphi(z)+\chi(z)$ with $\chi(z)=o((z-1)^2)$ and 
		$$|\chi(z)|\geq C_1|z-1|^\alpha.$$
		Then the $n^\mathrm{th}$ approximation number of $C_\varphi-C_\psi$ satisfies 
		$$a_n(C_\varphi-C_\psi)\gtrsim \frac{1}{(\log n)n^{\alpha-2}}.$$
	\end{thm} 
	\begin{proof}
		Let $n\geq 1$ and let us set 
		\begin{align*}
			z_j&=\frac{1+e^{i/(n-j)}}2,\ 1\leq j\leq n/2,\\
			w_j^{(1)}&=\varphi(z_j),\ 1\leq j\leq n/2,\\
			w_j^{(2)}&=\psi(z_j),\ 1\leq j\leq n/2,\\
			Z&=\{z_j:\ 1\leq j\leq n/2\},\\
			W&=\{w_j^{(1)}:\ 1\leq j\leq n/2\}\cup \{w_j^{(2)}:\ 1\leq j\leq n/2\}.
		\end{align*}
		
		We will soon see that $\textrm{card}(W)=2\lfloor n/2\rfloor$ provided that $n$ is large enough. The main step to apply Corollary \ref{cor:lowerbound} is to bound $\delta(W)$ from below. Almost all the difficulty in doing this is contained in the following fact.
		
		\smallskip
		
		\begin{fact}
			There exists $C>0$ (which does not depend on $n$) such that, for all $u\in\{1,2\}$ and for all $1\leq j\leq n/2$,
			$$\prod_{1\leq k\neq j\leq n/2}\rho(w_j^{(u)},w_k^{(1)})\rho(w_j^{(u)},w_k^{(2)})\geq C.$$
		\end{fact}
		
		\begin{proof}[Proof of the fact]
			We first observe for $k\in\{1,\dots,\lfloor n/2\rfloor \}$ that
			\begin{equation}\label{eq:lowerbound1}
				z_k=1+\frac i{n-k}-\frac1{4(n-k)^2}+o\left(\frac 1{(n-k)^2}\right).
			\end{equation}
			In particular, $\chi(z_k)=O(1/n^2)$ where the constant involved does not depend on $k$ since $k\leq n/2.$ Next let $v\in\{1,2\},$ $1\leq k\leq \lfloor n/2\rfloor $ and we write
			$$\varphi(z_k)=1+\varphi'(1)(z_k-1)+\frac{\varphi''(1)}{2}(z_k-1)^2+o((z_k-1)^2)$$
			with $\varphi'(1)\in(0,1].$ Therefore, taking into account \eqref{eq:lowerbound1}, we get 
			$$w_k^{(v)}=1+ia_{n,k}^{(v)}+b_{n,k}^{(v)}$$
			with $a_{n,k}^{(v)},b_{n,k}^{(v)}\in\mathbb R$ and 
			$$a_{n,k}^{(v)}=\frac{c_1}{n-k}+\frac{c_2}{(n-k)^2}+o\left(\frac 1{n^2}\right), \quad \ b_{n,k}^{(v)}=O\left(\frac 1{n^2}\right)$$
			for some $c_1\neq 0$ and $c_2\in\mathbb R.$ Then 
			$$|w_j^{(u)}-w_k^{(v)}|^2=(a_{n,j}^{(u)}-a_{n,k}^{(v)})^2+(b_{n,j}^{(u)}-b_{n,k}^{(v)})^2$$
			whereas 
			$$\overline{w_j^{(u)}}w_k^{(v)}=1+i(a_{n,k}^{(v)}-a_{n,j}^{(u)}+\delta_{n,j,k}^{(u,v)})+\gamma_{n,j,k}^{(u,v)}$$
			with $\delta_{n,j,k}^{(u,v)},\ \gamma_{n,j,k}^{(u,v)}\in\mathbb R$ and 
			$$\delta_{n,j,k}^{(u,v)}=O\left(\frac 1{n^3}\right), \quad \gamma_{n,j,k}^{(u,v)}=O\left(\frac 1{n^2}\right).$$
			This yields
			\begin{align*}
				\prod_{1\leq k\neq j\leq n/2}\rho^2(w_j^{(u)},w_k^{(v)})&=\prod_{1\leq k\neq j\leq n/2}\frac{(a_{n,j}^{(u)}-a_{n,k}^{(v)})^2+(b_{n,j}^{(u)}-b_{n,k}^{(v)})^2}{(a_{n,j}^{(u)}-a_{n,k}^{(v)})^2+2(a_{n,j}^{(u)}-a_{n,k}^{(v)})\delta_{n,j,k}^{(u,v)}+(\delta_{n,j,k}^{(u,v)})^2+(\gamma_{n,j,k}^{(u,v)})^2}\\
				&\geq \prod_{1\leq k\neq j\leq n/2}
				\frac1{1+2\frac{\delta_{n,j,k}^{(u,v)}}{a_{n,j}^{(u)}-a_{n,k}^{(v)}}+\frac{(\delta_{n,j,k}^{(u,v)})^2+(\gamma_{n,j,k}^{(u,v)})^2}{(a_{n,j}^{(u)}-a_{n,k}^{(v)})^2}}.
			\end{align*}
			For brevity we denote by $D_{n,j,k}^{(u,v)}$ the denominator of the preceding term. 
			
			Since 
			$$|a_{n,j}^{(u)}-a_{n,k}^{(v)}|\geq \frac{c_3|k-j|}{n^2}\geq \frac{c_3}{n^2},$$
			for $1\leq k\neq j\leq n/2$, we get that
			$$D_{n,j,k}^{(u,v)}\leq 1+\frac{c_4}n+\frac{c_4}{(k-j)^2}.$$
			The fact follows from the existence of an absolute constant $M>0$ such that
			$$\prod_{1\leq l\leq n}\left(1+\frac1{n}+\frac1{l^2}\right)\leq M.$$
		\end{proof}
		
		We are now ready to end the proof of Theorem \ref{thm:exlowerbound}. To estimate $\delta(W),$ we just need to consider $\rho(w_j^{(1)},w_j^{(2)}).$ The assumption on $\chi$ ensures that 
		$$|w_j^{(1)}-w_j^{(2)}|\geq \frac{c}{n^\alpha}$$
		for some $c>0$ independent of $n\geq 1$ and $1\leq j\leq n/2$, whereas 
		$$|1-\overline{w_j^{(1)}}w_j^{(2)}|\leq 1-|\varphi(z_j)|^2+|\varphi(z_j)|\cdot|\chi(z_j)|.$$
		The Taylor expansion of $\varphi$ reveals that 
		$$|\varphi(z_j)|^2=1+\frac{d_1}{(n-j)^2}+o\left(\frac 1{n^2}\right)$$
		so that
		$$0<1-|\varphi(z_j)|^2\leq \frac{d_2}{n^2}$$
		and
		$$|1-\overline{w_j^{(1)}}w_j^{(2)}|\leq\frac{d_3}{n^2}.$$
		This  yields that
		$$\rho(w_j^{(1)},w_j^{(2)})\gtrsim\frac1{n^{\alpha-2}}$$
		and finally $\delta(W)\gtrsim n^{-(\alpha-2)}.$ Applying the fact for $\varphi(z)=z,$ we also find $\delta(Z)\gtrsim 1.$ Since the computations above also yield
		$$\inf_{1\leq j\leq n}\left(\dfrac{1-|z_j|^2}{1-|\varphi(z_j)|^2}+\dfrac{1-|z_j|^2}{1-|\psi(z_j)|^2}\right)\gtrsim 1$$
		we have proven Theorem \ref{thm:exlowerbound}.
	\end{proof}
	\begin{example}
		Let $\varphi(z)=\frac{1+z}2$ and $\psi(z)=\varphi(z)+c(z-1)^\alpha$
		for $\alpha>2$ and $c\in(0,1/128).$ Then 
		$$a_n(C_\varphi-C_\psi)\gtrsim \frac 1{(\log n)n^{\alpha-2}}.$$
	\end{example}
	
	\subsection{Application to the difference of two maps with a corner}
	Let $$\varphi(z)=\frac1{1+(1-z)^{1/2}}$$
	and $\psi(z)=\varphi(z)+c \chi(z)$
	with 
	$$\chi(z)=\exp\left(-\frac 1{(1-z)^{1/2}}\right)$$
	for some sufficiently small $c\in(0,1].$ We intend to show that there exists $a>0$
	such that, for all $n\geq 2,$
	\begin{equation}\label{eq:lowerlen}
		a_n(C_\varphi-C_\psi)\gtrsim \exp\left(-a\frac n{\log n}\right).
	\end{equation}
	
	The proof will again follow from a suitable application of Corollary \ref{cor:lowerbound}. This time, we will choose a radial sequence $(z_j)$ which goes quickly to $1$.
	Let $\veps>0$, whose precise value will be given later, and we set
	\begin{align*}
		z_j&=1-e^{-j\veps},\ j\geq 1\\
		w_j^{(1)}&=\varphi(z_j),\ j\geq 1\\
		w_j^{(2)}&=\psi(z_j),\ j\geq 1.
	\end{align*}
	We will divide the proof into several lemmas.
	
	\begin{lem}\label{lem:lowerlen1}
		For all $\veps>0$, for all $n\geq 2$ and for all $j\in\{1,\dots,n\}$ we have that
		$$\frac{1-|w_{j+1}^{(1)}|}{1-|w_j^{(1)}|}\leq e^{-\veps/4}.$$
	\end{lem}
	\begin{proof}
		We start by writing
		\begin{align*}
			\frac{1-|w_{j+1}^{(1)}|}{1-|w_j^{(1)}|}=e^{-\veps/2}\frac{1+e^{-j\veps/2}}{1+e^{-(j+1)\veps/2}}
		\end{align*}
		and we set 
		$$f(y)=\frac{1+y}{1+ye^{-\veps/2}}$$
		with $y\in(0,e^{-\veps/2})$. Then 
		$$f'(y)=\frac{1-e^{-\veps/2}}{(1+ye^{-\veps/2})^2}>0$$
		so that $f$ is increasing. Therefore
		$$\frac{1-|w_{j+1}^{(1)}|}{1-|w_j^{(1)}|}\leq e^{-\veps/2}\frac{1+e^{-\veps/2}}{1+e^{-\veps}}\leq e^{-\veps/4}$$
		since $x (1+x)/(1+x^2)\leq \sqrt x$ for all $x\in(0,1).$
	\end{proof}
	
	We intend to generalise the previous lemma to $w_{j}^{(u)}$ for $u\in\{1,2\}.$
	For this we will use the following simple argument.
	
	\begin{lem}\label{lem:lowerlen2}
		There exists $\veps_0>0$ such that, for all $\veps\in(0,\veps_0),$ for all $n\geq 1$ and for all $j\geq j_0(\veps) \coloneqq \frac{|\log\veps|}{2\veps},$ 
		$$1-|w_j^{(2)}|\geq e^{-\veps/8}(1-|w_j^{(1)}|).$$
	\end{lem}
	
	\begin{proof}
		We observe that 
		$$1-|w_j^{(2)}|\geq 1-|w_j^{(1)}|-\exp(-\exp(\veps j/2))$$
		so that we only need to prove that 
		$$\exp(-\exp(\veps j/2))\leq (1-e^{-\veps/8}) (1-|w_j^{(1)}|)= (1-e^{-\veps/8}) \frac{e^{-\veps j/2}}{1+e^{-\veps j/2}}.$$
		Provided $\veps>0$ is small enough, $1-e^{-\veps/8}\geq\veps/16$ and it suffices to show that 
		\begin{equation}\label{eq:exlowerbound1}
			\exp(-\exp(\veps j/2))\leq \frac{\veps}{32}\exp(-\veps j /2).    
		\end{equation}
		Let us set $a=\veps/32$ and $x=\veps j/2.$ Then \eqref{eq:exlowerbound1} is
		equivalent to 
		$$x-\exp(x)\leq \log(a).$$
		Now, the function $g \colon y\mapsto y-\exp(y)$ is decreasing on $(0,+\infty)$
		and provided $a$ is small enough,
		$$g(\log(2|\log(a)|))=\log(2|\log a|)-2|\log(a)|\leq\log(a).$$
		Therefore, as soon as $j\geq 2\log(2|\log(\veps/32)|)/\veps,$
		which is true provided $j\geq j_0(\veps)$ for $\veps$ small enough,
		the desired inequality is satisfied.
	\end{proof}
	
	This enables us to write a general version of Lemma \ref{lem:lowerlen1}.
	\begin{lem} \label{lem:lowerlen1bis}
		There exists $\veps_0>0$ such that, for all $\veps\in(0,\veps_0),$ for all $n\geq 1,$ for all $j\geq j_0(\veps)\coloneqq\frac{|\log\veps|}{2\veps}$ and for all $u,v\in\{1,2\}$
		$$\frac{1-|w_{j+1}^{(u)}|}{1-|w_j^{(v)}|}\leq e^{-\veps/8}.$$
	\end{lem}
	\begin{proof}
		The result is already known if $u=v=1$. The general case follows by using the inequalities
		$$1-|w_j^{(1)}|\geq 1-|w_j^{(2)}|\geq e^{-\veps/8}(1-|w_j^{(1)}|).$$
	\end{proof}
	
	Next we consider $n\geq j_0(\veps)$ and we set 
	\begin{align*}
		Z&=\{z_j:\ j_0(\veps)\leq j\leq n\},\\
		W&=\{w_j^{(1)}:\ j_0(\veps)\leq j\leq n\}\cup \{w_j^{(2)}:\ j_0(\veps)\leq j\leq n\}.
	\end{align*}
	
	\begin{lem}\label{lem:lowerlen3}
		There exists $C>0$ such that, for all $\veps>0$ and all $n\geq j_0(\veps)$ 
		$$ \delta(W)\geq \exp\left(\frac{-C}{1-e^{-\veps/8}}\right)\exp\left(-\exp\left(\frac{n\veps}2\right)\right).$$ 
	\end{lem}
	\begin{proof}
		By \cite[Lemma 6.4]{LQRP12} and Lemma \ref{lem:lowerlen1bis}, we already know that there exists $C>0$ such that, for all $n\geq j_0(\veps)$ and for all $j_0(\veps)\leq j\leq n,$
		$$\prod_{j_0(\veps)\leq k\neq j\leq n}\rho(w_j^{(u)},w_k^{(1)})\rho(w_j^{(u)},w_k^{(2)})\geq \exp\left(\frac{-C}{1-e^{-\veps/8}}\right)$$
		where $u \in \{1,2\}$. 
		We conclude the proof with the easy estimate 
		$$\rho(w_j^{(1)},w_j^{(2)})\geq |w_j^{(1)}-w_j^{(2)}|\geq \exp\left(-\exp\left(\frac {n\veps}2\right)\right).$$
	\end{proof}
	
	We need a last lemma to apply Corollary \ref{cor:lowerbound}.
	
	\begin{lem}\label{lem:lowerlen4}
		For all $\veps>0,$ for all $n\geq 1$, for all $1\leq j\leq n$ and for all $u\in\{1,2\}$
		$$\frac{1-|z_j|^2}{1-|w_j^{(u)}|^2}\geq \frac 12\exp\left(-\frac{n\veps}2\right).$$
	\end{lem}
	\begin{proof}
		It suffices to observe that 
		$$1-|z_j|=e^{-j\veps}$$
		and that 
		$$1-|w_j^{(u)}|\leq 1-|w_j^{(1)}|=\frac{e^{-\veps j/2}}{1+e^{-\veps j}}\leq e^{-\veps j/2}.$$
	\end{proof}

	We are now ready to prove \eqref{eq:lowerlen}.
	Observe that we also have 
	$$\delta(Z)\geq \exp\left(\frac{-C}{1-e^{-\veps/8}}\right)$$
	by another application of, for instance, \cite[Lemma 6.4]{LQRP12}.
	We choose $\veps=\log(n)/n$ so that 
	$n=|\log\veps|/\veps+o(|\log\veps|/\veps)$. We set $p=\lfloor n-j_0(\veps)\rfloor$ and observe that $p\geq n/3$ provided $n$ is large enough.
	
	Corollary \ref{cor:lowerbound} and the preceding lemmas give the existence of some constant $c_1>0$ such that 
	$$a_p(C_\varphi-C_\psi)\gtrsim \exp\left(\frac{-c_1}{1-e^{-\veps/8}}\right)
	\exp\left(-c_1\exp\left(\frac{n\veps}2\right)\right)\exp(-c_1n\veps).$$
	The Taylor expansion of $\exp(-\veps/8)$ yield the existence of $c_2>0$ such that
	$$a_p(C_\varphi-C_\psi)\gtrsim \exp\left(-\frac{c_2}\veps \right)
	\exp\left(-c_2\exp\left(\frac{n\veps}2\right)\right)\exp(-c_2n\veps)$$.
	We get the desired result by our choice of $\veps$. \hfill $\square$

	\section{Upper bounds for the Approximation Numbers of the Difference of Composition Operators}
	
	\subsection{A general upper bound}
	
	We turn now to the problem of obtaining upper bounds for the approximation numbers of differences of composition operators.  We will first give a general upper bound in terms of Blaschke products.
	
	Throughout this section the pseudohyperbolic disc centered at $z_0$ with radius $r\in(0,1)$ will be denoted by $\Delta(z_0,r)$, and for brevity we will sometimes denote the Hardy space $H^2(\DD)$ by $H^2$. We shall require the following lemma (cf.~\cite[Lemma 5.1]{Sau}).
	\begin{lem}\label{lem:saukkho}
		Let $r\in (0,1)$. There exists a constant $C=C(r)>0$ such that, for all $f\in H^1(\DD)$, for all $z_0\in \DD$, for all $z\in \Delta(z_0,r)$,
		$$|f(z)-f(z_0)|\leq C \rho(z,z_0)P(|f|)(z_0).$$
	\end{lem}
	
	\begin{thm}\label{thm:difference}
		Let $r\in(0,1)$ and for $n\geq 1$ let $B$ be a Blaschke product of degree $n-1$. 
		Further assume that $\varphi$ and $\psi$ are holomorphic self-maps of $\DD$, and let the function $w \colon \DD\to(0,1)$ be given by $z\mapsto \rho(\varphi(z),\psi(z))$. Then
		\begin{align*}
			a_n(C_\varphi-C_\psi) &\lesssim \left( \sup_{\{z\in\TT \; :\ |\varphi(z)|<r\}} |B(\varphi(z))|+\sup_{\{z\in\TT \; :\ |\psi(z)|<r\}} |B(\psi(z))|\right. 			\\
			& \left. \qquad\quad + 
			\sup_{\{z\in\TT \; :\ |\varphi(z)|>r\}} |w(z)|+ \sup_{\{z\in\TT \; :\ |\psi(z)|>r\}} |w(z)| \right) \times\left(\|C_\varphi\|+\|C_\psi\|\right).
		\end{align*}
	\end{thm}
	\begin{proof}
		We may assume that $\varphi \not\equiv \psi$. From the definition of the $n^\mathrm{th}$ approximation number, we know that $a_n(C_\varphi-C_\psi)\leq \|C_\varphi-C_\psi-R_{n-1}\|$ for any operator $R_{n-1} \colon H^2 \to H^2$ of rank less than $n-1$. 
		
		We will use the following rank $n-1$ operator.  Let $P_B$ denote the orthogonal projection from $H^2$ to the model space $K_{B}^{2}=H^2\ominus BH^2$ and let $R_{n-1}=(C_\varphi-C_\psi)\circ P_B$. Let $f\in H^2(\DD)$ with $\|f\|\leq 1$ and let us observe that, for any $z\in\DD$,
		$$f(z)-P_Bf(z)=B(z)F(z)$$
		for some $F\in H^2(\DD)$ with $\|F\|_2\leq \|f\|_2\leq 1.$ In particular, 
		$$G\coloneqq(C_\varphi-C_\psi)(f)-R_{n-1}(f)=B\circ\varphi\cdot F\circ\varphi-B\circ\psi\cdot F\circ\psi.$$
		We set 
		$$\Omega_0=\{z\in\TT:\ \varphi(z)\textrm{ and }\psi(z)\textrm{ exist and }\varphi(z)\neq \psi(z)\}$$
		and observe that the complement of $\Omega_0$ has measure $0.$
		We then set 
		\begin{align*}
			\Omega_1&=\{z\in\Omega_0:\ |\varphi(z)|<r\textrm{ and }|\psi(z)|<r\}\\
			\Omega_2&=\{z\in\Omega_0:\ |\varphi(z)|>r\}\\
			\Omega_3&=\{z\in\Omega_0:\ |\psi(z)|>r\}.
		\end{align*}
		We next evaluate the following integral over $\Omega_1$, with respect to the Lebesgue measure $\sigma$:
		\begin{align*}
			\int_{\Omega_1} |G|^2d\sigma&\lesssim \sup_{z\in\Omega_1}|B(\varphi(z))|^2 \int_{\Omega_1}|F\circ\varphi(z)|^2d\sigma(z)+\sup_{z\in\Omega_1} |B(\psi(z))|^2\int_{\Omega_1}|F\circ\psi(z)|^2d\sigma(z)\\
			&\leq \sup_{\{z\in\TT:\ |\varphi(z)|<r\}}|B(\varphi(z))|^2 \|C_\varphi\|^2+\sup_{\{z\in\TT:\ |\psi(z)|<r\}} |B(\psi(z))|^2 \|C_\psi\|^2.
		\end{align*}
		To evaluate the integral over $\Omega_2$, we split the set into two parts: 
		\begin{align*}
			\Omega_2'&=\{z\in\Omega_2:\ |w(z)|> 1/2\}, \\
			\Omega_2''&=\{z\in\Omega_2:\ |w(z)|\leq 1/2\}.
		\end{align*}
		On the one hand,
		\begin{align*}
			\int_{\Omega_2'}|G|^2d\sigma&\lesssim \int_{\Omega'_2} |w(z)|^2 |B(\varphi(z))|^2 |F(\varphi(z))|^2d\sigma(z)+\int_{\Omega'_2} |w(z)|^2 |B(\psi(z))|^2 |F(\psi(z))|^2d\sigma(z)\\
			&\lesssim \sup_{|\varphi(z)|>r} |w(z)|^2 \left(\int_{\TT}|F\circ\varphi|^2d\sigma+\int_{\TT}|F\circ\psi|^2d\sigma\right)\\
			&\lesssim \sup_{|\varphi(z)|>r}|w(z)|^2\left(\|C_\varphi\|^2+\|C_\psi\|^2\right).
		\end{align*}
		On the other hand, for any $z\in\Omega_2''$, $|\varphi(z)|<1$ and $\psi(z)\in\Delta(\varphi(z),1/2)$, we apply Lemma \ref{lem:saukkho} which yields
		\begin{align*}
			\int_{\Omega_2''}|G|^2 d\sigma&\lesssim \int_{\Omega_2''} |w(z)|^2 |P(|BF|)(\varphi(z))|^2d\sigma(z)\\
			&\lesssim \sup_{|\varphi(z)|>r} |w(z)|^2 \|BF\|_2^2 \|C_\varphi\|^2 \lesssim \sup_{|\varphi(z)|>r} |w(z)|^2 \|C_\varphi\|^2
		\end{align*}
		where we have used Theorem \ref{thm:embedding}. The integral over $\Omega_3$ gives rise to similar inequalities, and which allows us to obtain the statement of Theorem \ref{thm:difference}.
	\end{proof}
	
	Our next task is to find a suitable Blaschke product. The choice $B(z)=z^n$ already gives interesting results, however a more intriguing approach is to adapt the choice of $B$ to the maps $\varphi$ and $\psi.$
	We will make use of the work of \cite[Section 3.2]{QS}.
	
	\begin{lem}\label{lem:blasckeQS}
		Let $\varphi$ be a holomorphic self-map of $\DD$ and for $r\in(0,1),$ we set $\mathcal E_r=\{z\in\TT:\ |\varphi(z)|\leq r\}$ and $\Omega_r=\varphi(\mathcal E_r).$ Assume that $\varphi$ is differentiable on $\mathcal E_r$ and that $\Omega_r$ is connected for every $r\in (0,1)$. Then for each $r\in(0,1)$, there exists a Blaschke product $B$ of degree $n-1$ such that 
		$$\sup_{z\in \mathcal E_r}|B(\varphi(z))|\leq \exp(-(\pi^2/2+o(1))n/\ell_P(\Omega_r)) \; \textrm{ when } \; r\to1.$$
	\end{lem}

	\subsection{Application to the difference of two non-compact composition operators}
	
	The next theorem should be thought of as a counterpart to Theorem \ref{thm:exlowerbound} and it leads immediately to the second half of Example \ref{ex:differencesmooth}.
	
	\begin{thm}\label{thm:moorhouselike}
		Let $\varphi$ and $\psi$ be holomorphic self-maps of $\DD$ and let $\alpha>2$. Suppose that 
		\begin{enumerate}[label=(\roman*)]
			\item $\varphi$ and $\psi$ extend to a $C^1$-function on $\TT$ with $\varphi(1)=\psi(1)=1$ and $|\varphi(z)|,|\psi(z)|<1$ if $z\neq 1$.
			\item There exist $c_1,c_2,C>0$ such that, for all $t\in(-\pi,\pi]$
				\begin{align*}
					c_1 t^2 &\leq 1-|\varphi(e^{it})|^2\leq c_2t^2 \\
					c_1 t^2 &\leq1-|\psi(e^{it})|^2\leq c_2t^2 \\
					|\varphi(e^{it})-\psi(e^{it})| &\leq C |e^{it}-1|^\alpha.
				\end{align*}
		\end{enumerate}
		Then the $n^\mathrm{th}$ approximation number of the difference operator $C_\varphi-C_\psi$ satisfies
		$$a_n(C_{\varphi}-C_{\psi})\lesssim \left(\frac{\log n}{n}\right)^{\alpha-2}.$$
	\end{thm}	    
	\begin{proof}
		Let $r\in (0,1)$ and $n\in\mathbb N$. Let $B_1$ be the Blaschke product of degree $n-1$ associated with $\varphi$ and let $B_2$
		be the Blaschke product of degree $n-1$ associated with $\psi$, as given by Lemma \ref{lem:blasckeQS}. We set $B=B_1B_2$ and observe that
		\begin{align*}
			\sup_{|\varphi(z)|<r} |B(\varphi(z))|&\leq \sup_{|\varphi(z)|<r} |B_1(\varphi(z))|\\
			&\leq \exp\left(-\left(\frac{\pi^{2}} 2+o(1)\right)n/\ell_P(\Omega_r)\right)
		\end{align*}
		where $\Omega_r=\varphi\left(\left\{z\in\TT:\ |\varphi(z)|<r\right\}\right)$.  Now observe that for $t\in(-\pi,\pi]$, $|\varphi(e^{it})|<r\implies |t|\geq c\sqrt{1-r}$ for some constant $c>0$. Therefore
		$$\Omega_r\subset\varphi\left(\left\{e^{it}:\ |t|\geq c\sqrt{1-r}\right\}\right).$$
		Let us denote, for $t\in(-\pi,\pi]$, $z(t)=\varphi(e^{it})$ and let us observe that $|z'(t)|\lesssim 1$ since $\varphi$ is $C^1$ 
		on $\overline{\DD}.$ Therefore
		\begin{align*}
			\ell_P(\Omega_r)&\leq \int_{c\sqrt{1-r}}^{\pi}\frac{|z'(t)|}{1-|z(t)|^2}dt+\int_{-\pi}^{-c\sqrt{1-r}}\frac{|z'(t)|}{1-|z(t)|^2}dt\\
			&\lesssim \int_{c\sqrt{1-r}}^{\pi}\frac{dt}{t^2}\lesssim \frac 1{\sqrt{1-r}}.
		\end{align*}
		On the other hand, let $z=e^{it}$ be such that $|\varphi(z)|>r.$
		Then one can observe that $|t|\lesssim \sqrt{1-r}$. Since 
		$$|z-1|^\alpha=|e^{it}-1|^\alpha\lesssim |t|^\alpha$$
		and since
		\begin{equation*}
			|1-\overline{\varphi(z)}\psi(z)| \geq 1-|\varphi(z)|^2-C|\varphi(z)|\cdot |1-z|^\alpha \gtrsim t^{2}
		\end{equation*}
		it follows that
		$$|w(z)|\lesssim \frac{|z-1|^\alpha}{|1-\overline{\varphi(z)}\psi(z)|}\lesssim |t|^{\alpha-2}\lesssim (1-r)^{\frac\alpha 2-1}.$$
		By a similar argument for the terms involving $\psi$, we finally get
		by Theorem \ref{thm:difference} that
		$$a_{2n-1}(C_\varphi-C_\psi)\lesssim \exp(-Cn\sqrt{1-r})+(1-r)^{\frac\alpha 2-1}.$$
		We conclude by choosing $\sqrt{1-r}=D(\log n/n)^2$ for some sufficiently large constant $D.$
	\end{proof}
	
	We note that, under the previous assumptions, both $C_\varphi$ and $C_\psi$ are non-compact since $\varphi'(1)$ and $\psi'(1)$ exist. 
	
	\medskip

	Combining Theorems \ref{thm:exlowerbound} and \ref{thm:moorhouselike}
	we get the following corollary.
	\begin{cor}
		Let $\varphi$ and $\psi$ be holomorphic self-maps of $\DD$ and let $\alpha>2$. Suppose that 
		\begin{enumerate}[label=(\roman*)]
			\item[(i)] $\varphi$ and $\psi$ extend to a $C^2$-function on $\TT$ with $\varphi(1)=\psi(1)=1$ and $|\varphi(z)|,|\psi(z)|<1$ if $z\neq 1$.
			\item[(ii)] There exist $c_1,c_2>0$ such that, for all $z\in\overline\DD,$ 
			$$c_1 |z-1|^\alpha\leq |\varphi(z)-\psi(z)|\leq c_2|z-1|^\alpha.$$
		\end{enumerate}
		Then the $n^\mathrm{th}$ approximation number of the difference operator $C_\varphi-C_\psi$ satisfies
		$$\frac{1}{(\log n)n^{\alpha-2}}\lesssim a_n(C_{\varphi}-C_{\psi})\lesssim \left(\frac{\log n}{n}\right)^{\alpha-2}.$$
	\end{cor}
	
	One may apply this corollary for $\alpha=k+1$ if $\varphi_0,\varphi_1$ are $C^{k+1}$ on $\mathbb T$ for some $k\geq 2$, they satisfy (i) and it holds that $\varphi_{0}(1)=\varphi_{1}(1), \varphi'_{0}(1)=\varphi'_{1}(1),\dots, \varphi^{(k)}_{0}(1)=\varphi^{(k)}_{1}(1) $
	(we say that $\varphi_0$ and $\varphi_1$ have the same data at  $\{1\}$ up to $k^\mathrm{th}$ order). 
	In line with the results of \cite{MT} and \cite{BOU}, this gives further information on the link between the degree of data similarity and compactness. 
	
	\subsection{Application to the difference of two maps with a corner}
	
	We now revisit the choice $$\varphi(z)=\frac1{1+(1-z)^{1/2}}$$
	and $\psi(z)=\varphi(z)+c \chi(z)$
	with 
	$$\chi(z)=\exp\left(-\frac 1{(1-z)^{1/2}}\right)$$
	for some sufficiently small $c>0.$ We intend to show that there exists $b>0$
	such that, for all $n\geq 2,$
	\begin{equation}\label{eq:upperlen}
		a_n(C_\varphi-C_\psi)\lesssim \exp\left(-b\frac n{\log n}\right).
	\end{equation}
	Let $\delta\in(0,1),$ $r=1-\delta$ and let us set 
	$$\Omega_r=\varphi\left(\left\{z\in\TT:\ |\varphi(z)|\leq r\right\}\right),\quad \Omega'_r=\psi\left(\left\{z\in\TT:\ |\psi(z)|\leq r\right\}\right).$$
	
	We first show that $\ell_P(\Omega_r),\ell_P(\Omega'_r)\lesssim |\log(\delta)|$. Let us prove it for the most difficult case, namely $\Omega'_r.$ We set $z(t)=\psi(e^{it}).$ As we have already observed in Section \ref{sec:map}, 
	$$|z(t)|^2=1-\sqrt 2 |t|^{1/2}+o(|t|^{1/2})$$
	so that the condition $|\psi(e^{it})|\leq r$ implies $|t|\geq c_1\delta^2$
	for some constant $c_1>0.$
	Computing the derivative we find that
	$$z'(t)=\frac{-ie^{it}}{2(1+(1-e^{it})^{1/2})^2(1-e^{it})^{1/2}}+c
	\frac{-ie^{it}}{2(1-e^{it})^{3/2}}\exp\left(-\frac1{(1-e^{it})^{1/2}}\right)$$
	and via a Taylor expansion we easily get that
	$$|z'(t)|\lesssim \frac 1{|t|^{1/2}}.$$
	We finally find 
	$$\ell_P(\Omega'_r)\lesssim \int_{c_1\delta^2}^{\pi}\frac{dt}{|t|^{1/2}\cdot |t|^{1/2}}\lesssim |\log(\delta)|.$$
	For $n\in\mathbb N$, we let $B_1$ and $B_2$ denote, respectively, the Blaschke products of degree $n$ associated to $\varphi$ and $\psi$, as given by Lemma \ref{lem:blasckeQS}. 
	We have thus obtained that
	$$\sup_{|\varphi(z)|<r}|B(\varphi(z))|, \sup_{|\psi(z)|<r}|B(\psi(z))|\leq \exp(-c_2n/|\log(\delta)|)$$
	for some $c_2>0.$ 
	
	In order to apply Theorem \ref{thm:difference}, it remains to estimate $\sup_{|\varphi(z)|>r}w(z)$ and $\sup_{|\psi(z)|>r}w(z)$. In view of the previous estimate, we just need to concentrate on $\sup_{|t|\leq c_3\delta^2}w(e^{it}).$
	We note that 
	\begin{equation*}
		|1-\overline{\varphi}(e^{it})\psi(e^{it})| \geq 1-|\varphi(e^{it})|^2-|\varphi(e^{it})|\cdot|\chi(e^{it})| \geq c_4 |t|^{1/2}
	\end{equation*}
	whereas 
	$$|\varphi(e^{it})-\psi(e^{it})|\lesssim \exp(-c_5/|t|^{1/2}).$$
	Finally,
	$$a_{2n-1}(C_\varphi-C_\psi)\lesssim \exp(-c_2n/|\log(\delta)|)+\exp(-c_6/\delta).$$
	The choice $\delta=\log(n)/n$ thus yields \eqref{eq:upperlen}. $\hfill\square$

	\section{Approximation Numbers of weighted Composition Operators}
		
	We turn now to determine estimates of the approximation numbers of weighted composition operators $M_\omega C_\varphi \colon g \mapsto \omega \cdot (g\circ \varphi)$. We will improve several results of \cite{LLQRP} and we begin with the upper bounds.

		\begin{thm}\label{thm:upperweighted}
			Let $\varphi \colon \DD\to\DD$ be holomorphic and $\omega\in H^\infty(\DD)$.
			For $n\in\mathbb N$ and $r\in(0,1)$, let $B$ be an arbitrary Blaschke product of degree $n-1$. Then 
			\begin{equation}\label{eq:appnumweight}
				a_n(M_\omega C_\varphi)\leq \left(\sup_{\{z\in\mathbb{T}:|\varphi(z)|\leq r\}}|B(\varphi(z))|^2\|T\|^2+\delta_{0}^{2}(r)\|C_\varphi\|^2\right)^{1/2}
			\end{equation}
			where $\delta_0(r)=\sup_{\{t:|\varphi(e^{it})|>r\}}|\omega(\varphi(e^{it}))|.$
		\end{thm}
		\begin{proof}
			Let $B$ be an arbitrary Blaschke product of degree $n-1$ and let $K_{B}^{2}$ be the model space defined as $K_{B}^{2}=H^2\ominus BH^2$. Let $P_B$ denote the orthogonal projection from $H^2$ to $K_{B}^{2}$ and set $R_{n-1}=T\circ P_B$. Suppose also that $f\in H^2(\mathbb{D})$ with $\|f\|_{H^2}=1$ then 
			$$
			\|(T-R_{n-1})f\|_{H^2}^{2}=\int_{\mathbb{T}}|G(z)|^2d\sigma(z)
			$$  
			where 
			\begin{equation*}
				G(z) \coloneqq \omega(z)(f(\varphi(z))-P_Bf(\varphi(z)))=\omega(z)B(\varphi(z))F(\varphi(z))
			\end{equation*}
			 and $\|F\|_{H^2}\leq \|f\|_{H^2}=1$.
			Then for $0<r<1$,
			\begin{align*}
				\|(T-R_{n-1})f\|_{H^2}^{2}&\leq \sup_{\{z\in\mathbb{T}:|\varphi(z)|\leq r\}}|B(\varphi(z))|^2\|T\|^2+ \int_{\{z\in\mathbb{T}:|\varphi(z)|>r\}}|\omega_0\circ\varphi(z)|^2|F(\varphi(z))|^2d\sigma(z)\\
				&\leq  \sup_{\{z\in\mathbb{T}:|\varphi(z)|\leq r\}}|B(\varphi(z))|^2\|T\|^2 + \delta_0(r)^2\|C_\varphi\|^2.
			\end{align*}
		\end{proof}
	For a smooth symbol at the boundary, we get the following corollary:	
		\begin{cor}
			Let $\varphi$ be a holomorphic self-map of $\DD$ and let $\omega\in H^\infty(\DD)$. Suppose that 
			\begin{enumerate}[label=(\roman*)]
				\item $\varphi$ extends to a $C^1$-function on $\TT$ with $\varphi(1)=1$ and $|\varphi(z)|<1$ if $z\neq 1$.
				\item There exist $c_1,c_2>0$ such that, for all $t\in(-\pi,\pi],$
				$$c_1 t^2\leq 1-|\varphi(e^{it})|^2\leq c_2 t^2.$$
				\item There exist $c_2,\alpha>0$ such that, for all $z\in\DD,$ 
				$$|\omega(z)|\leq c_3 |z-1|^\alpha.$$
			\end{enumerate}
			Then 
			$$
			a_n(M_{\omega}C_{\varphi})\lesssim \left(\dfrac{\log(n)}{n}\right)^{\alpha}.
			$$
		\end{cor} 
		
		\begin{proof}
			
			Let $n\in\mathbb N,$ $r\in(0,1)$ and let $B$ be the Blaschke product 
			associated to $\varphi$ by Lemma \ref{lem:blasckeQS}. Arguing as in 
			the proof of Theorem \ref{thm:moorhouselike} we find that
			$$
			\sup_{|\varphi(z)|<r}|B(\varphi(z))|\leq  \exp(-cn\sqrt{1-r}) 
			$$
			for some $c>0$.
			Moreover, provided $|\varphi(e^{it})|>r$,
			\begin{equation*}
				|\omega(\varphi(e^{it}))| \lesssim |\varphi(e^{it})-1|^\alpha \lesssim |t|^\alpha \lesssim (1-r)^{\alpha/2}.
			\end{equation*}
			Therefore, applying Theorem \ref{thm:upperweighted} with  $\sqrt{1-r}=\frac{C\log(n)}{n}$ for a sufficiently large $C>0$  we obtain the result.
		\end{proof}

		\begin{rem}
			The preceding result applied to $\varphi(z)=(1+z)/2$ and $\omega(z)=(1-z)^{\alpha}$ improves \cite[Theorem 2.3]{LLQRP}), where the Blaschke product used was $B(z)=z^n$.
		\end{rem}

		We now turn our attention to the lower bounds for the approximation number of weighted composition operators. We begin by improving \cite[Lemma 2.6]{LLQRP}.
		
		\begin{thm}\label{thm:lowerboundweighted}
 Let $\varphi$ be a holomorphic self-map of $\DD$ and let $\omega\in H^\infty(\DD)$. Let $n\geq 1,$ let $Z=\{z_1,\dots, z_n\}$ and let $W=\varphi(Z)$ be such that $\textrm{card}(W)=n.$ Then  
			$$
			a_n(M_\omega C_\varphi)\geq M(W)^{-1}\|v_Z\|_{\mathcal{C}}^{-1/2}\inf_{1\leq j\leq n}\left(|\omega(z_j)|^2\dfrac{1-|z_j|^2}{1-|\varphi(z_j)|^2}\right)^{1/2}.
			$$   
			In particular, 
			$$a_n(M_\omega C_\varphi)\gtrsim \frac{\delta(W)}{\left(1+\log\left(\frac 1{\delta(W)}\right)\right)^{1/2}\left(1+\log\left(\frac 1{\delta(Z)}\right)\right)^{1/2}
			}\inf_{1\leq j\leq n}\left(|\omega(z_j)|^2\dfrac{1-|z_j|^2}{1-|\varphi(z_j)|^2}\right)^{1/2}$$
		\end{thm}
		
		\begin{proof}
			We follow the proof of Theorem \ref{thm:lowerbound}.
			Let $E(Z)=\mathrm{span}\{k_{z_1},\dots,k_{z_n}\}$ and pick $f\in E(Z)$, $f=\sum_{j}b_jk_{z_j}$ with $\|f\|_2=1$. Then
			\begin{align*}         
				\|(M_\omega C_{\varphi})^{*}(f)\|_{2}^{2}&=	\|\sum_{j}b_j\overline{\omega(z_j)}k_{\varphi(z_j)}\|_{2}^{2}\\
				&
				\geq M(W)^{-2}\sum_{j}|b_j|^2\dfrac{|\omega(z_j)|^2}{1-|\varphi(z_j)|^2}
				\\
				&\geq 
				M(W)^{-2}\left(\inf_{j}|\omega(z_j)|^2\dfrac{1-|z_j|^2}{1-|\varphi(z_j)|^2}\right)\sum_{j}\dfrac{|b_j|^2}{1-|z_j|^2}\\
				&\geq M(W)^{-2}\|\nu_Z\|_{\mathcal C}^{-1}\inf_{j}\left(|\omega(z_j)|^2\dfrac{1-|z_j|^2}{1-|\varphi(z_j)|^2}\right)
			\end{align*}
		\end{proof}    
		
		\begin{cor}
			Let $\varphi \colon \DD\to\DD$ be holomorphic and which extends to a function belonging to $\mathcal C^2(\{1\})$. Let $\omega\in H^\infty(\DD)$. If there exist $\alpha>0$ and $C>0$ such that $|\omega(z)|\geq C_1|1-z|^\alpha$ for all $z\in\DD$, then
			 $$a_n(M_\omega C_\varphi)\gtrsim\frac1{n^\alpha}.$$
		\end{cor}
		
		\begin{proof}
			For $n\geq 1,$ we set
			$$z_j=\frac{1+e^{i/(n-j)}}2,\ 1\leq j\leq n/2.$$
			The computations made during the proof of Theorem \ref{thm:exlowerbound} show that $\delta(W),\delta(Z)\gtrsim 1$ and that
			$$\inf_{1\leq j\leq n}\frac{1-|z_j|^2}{1-|\varphi(z_j)|^2}\gtrsim 1.$$
			The result follows from the fact that
			$$|\omega(z_j)|\gtrsim \frac 1{n^\alpha}.$$
		\end{proof}
		
		Of course, we may apply the previous corollary to $\varphi(z)=(1+z)/2$ and $\omega(z)=(1-z)^\alpha$ to conclude that 
		$$\frac 1{n^\alpha}\lesssim a_n(M_\omega C_\varphi)\lesssim \frac{\log(n)}{n^\alpha}.$$
		
		\section{Applications to Bidisc Case}

		Throughout this section we will use the one-dimensional techniques and examples that we have hitherto considered to understand the approximation numbers of the differences of composition operators with symbols defined on the bidisc $\mathbb{D}^2$.
		
		\subsection{Split Symbols} We will start with the simple case of split symbols. Let $\Phi_0  \colon \mathbb{D}^2\rightarrow\mathbb{D}^2$ and $\Phi_{1} \colon \mathbb{D}^2\rightarrow\mathbb{D}^2$ be  symbols given by
		\begin{align*}
			\Phi_{0}&=(\phi_0,\psi) \colon \mathbb{D}^2\rightarrow\mathbb{D}^2, \\
			\Phi_{1}&=(\phi_1,\psi)  \colon \mathbb{D}^2\rightarrow\mathbb{D}^2
		\end{align*}
		where the symbols $\phi_0,\phi_1  \colon \mathbb{D}\rightarrow\mathbb{D}$ depend only on $z_1$ and $\psi  \colon \mathbb{D}\rightarrow\mathbb{D}$ depends only on $z_2$.
		
		For the sake of completeness, we first recall some basic facts concerning tensor products. Let $f,g\in H^2(\mathbb{D})$, then for $z=(z_1,z_2) \in \DD^2$,
		\begin{align*}
			\left( C_{\Phi_0}-C_{\Phi_{1}} \right) (f\otimes g) &=f(\phi_0(z_1))g(\psi(z_2))-f(\phi_1(z_1))g(\psi(z_2))\\
			&=(f(\phi_0(z_1))-f(\phi_1(z_1)))g(\psi(z_2))=((C_{\phi_0}-C_{\phi_1})(f)\otimes C_{\psi}(g))(z)
		\end{align*}
		and since $H^2(\mathbb{D}^2)=H^2(\mathbb{D})\otimes H^2(\mathbb{D})$ we have 
		$$
		C_{\Phi_0}-C_{\Phi_{1}}=(C_{\phi_0}-C_{\phi_1})\otimes C_{\psi}.
		$$
		 
		 Moreover, for the approximation numbers of tensor products we have the following result \cite[Lemma 3.2]{LQRP19}.
		 \begin{lem}
		 	Let $S \colon H_1\rightarrow H_1$ and $T \colon H_2\rightarrow H_2$ be two compact linear operators where $H_1$ and $H_2$ are Hilbert spaces. If $S\otimes T$ is the tensor product acting on $H_1\otimes H_2$, then
		 	$$
		 	a_{mn}(S\otimes T)\geq a_m(S)a_n(T)
		 	$$
		 	for all positive integers $m,n$.
		 \end{lem}
\begin{cor}
Under the above assumptions,
$$a_n(C_{\phi_0}-C_{\phi_1})\geq a_{\sqrt n}(C_{\phi_0}-C_{\phi_1})a_{\sqrt n}(C_\psi).$$
\end{cor}	
We may now use all the examples given in the present paper and in the litterature on composition operators to get interesting examples on the bidisc.

		\begin{example}\label{eg:glued}
			Let $\varphi_0(z)=\dfrac{1}{1+(1-z)^{1/2}}$, $\varphi_1(z)=\varphi_0(z)+c\chi(z)$ where 
			\begin{equation*}
				\chi(z)=\exp\left(-\dfrac{1}{(1-z)^{1/2}}\right)
			\end{equation*} 
			for sufficiently small $c>0$ and let $\psi$ be the cusp map, so that $a_n(C_\psi)\lesssim e^{-\alpha n/\log n}$ by \cite[Theorem 4.3]{LQRP13}. Then for some $a>0$,
			$$
			a_n(C_{\Phi_0}-C_{\Phi_{1}})\gtrsim \exp\left(-a\dfrac{\sqrt{n}}{\log(n)}\right).
			$$
		\end{example}
		
		\subsection{Glued Symbols} Let $\varphi,\psi \colon \mathbb{D}\rightarrow\mathbb{D}$ be nonconstant analytic maps. Define 
		\begin{align*}
					\Phi_{\varphi}(z_1,z_2) &=(\varphi(z_1),\varphi(z_1)) \\
			\Phi_{\psi}(z_1,z_2) &=(\psi(z_1),\psi(z_1))
		\end{align*}
		then $\Phi_{\varphi}$ and $\Phi_{\psi}$ are \emph{glued symbols} in the terminology of \cite{LQRP19}.
As observed in \cite{LQRP19}, if we define the subspace
			$$
			E=\left\{f\in H^2(\mathbb{D}^2):\dfrac{\partial f}{\partial z_2}\equiv 0\right\},
			$$
we note that $E$ is isometrically isomorphic to $H^2(\mathbb{D})$. When we restrict the operators $C_{\Phi_{\varphi}}$ and $C_{\Phi_{\psi}}$ to $E$, then by the definition of glued symbols we obtain that their behaviour is exactly that of $C_\varphi$ and $C_\psi$ on $H^2(\mathbb{D})$ and we get
			\begin{align*}
				a_n(C_{\Phi_{\varphi}}-C_{\Phi_{\psi}})&\geq a_n([C_{\Phi_{\varphi}}-C_{\Phi_{\psi}}]_{\vert E}).
			\end{align*}
Again we may use the specific examples of the present paper.
			
		
		\subsection{Triangularly Separated Symbols}
		Before introducing triangularly separated symbols, let us first give a simple observation about the approximation numbers of the differences of weighted composition operators which will be useful in the below example.
		
		\begin{prop}
			Let $u_0,u_1\in H^\infty(\mathbb{D})$ be weights satisfying $\|u_0\|\leq 1$, $\|u_1\|\leq 1$ and let $\varphi_0,$ $\varphi_1$ be holomorphic self-maps of $\mathbb{D}$. Then
			\begin{align*}
				a_n(M_{u_0}C_{\varphi_0} - M_{u_1}C_{\varphi_1}) \leq \min \left( \right. &\|u_0\|_{\infty} a_n(C_{\varphi_0}-C_{\varphi_1})+\|u_0-u_1\|_{\infty}a_n(C_{\varphi_1}), \\
				& \left.  \quad \|u_1\|_{\infty} a_n(C_{\varphi_0}-C_{\varphi_1})+\|u_0-u_1\|_{\infty}a_n(C_{\varphi_0}) \right).
			\end{align*}
		\end{prop}
		\begin{proof}
			Let $E$ be a subspace of $H^2(\mathbb{D})$ with $\mathrm{codim}(E)<n$ and $f\in E$, then
			\begin{align*}
				\|u_0C_{\varphi_0}(f)-u_1C_{\varphi_1}(f)\|_{E}^{2} &=\int_{0}^{2\pi}|u_0C_{\varphi_0}(f)-u_1C_{\varphi_1}(f)|^2d\theta \\
				&=\int_{0}^{2\pi}|u_0(C_{\varphi_0}(f)-C_{\varphi_1}(f))+C_{\varphi_1}(f)(u_0-u_1)|^2d\theta \\
				&\leq \|u_0\|_{\infty}^{2}\int_{0}^{2\pi}|C_{\varphi_0}(f)-C_{\varphi_1}(f)|^2d\theta+\|u_0-u_1\|_{\infty}^{2}\int_{0}^{2\pi}|C_{\varphi_1}(f)|^2d\theta
			\end{align*}
			which gives that 
			\begin{equation*}
				a_n(M_{u_0}C_{\varphi_0}-M_{u_1}C_{\varphi_1})\leq \|u_0\|_{\infty} a_n(C_{\varphi_0}-C_{\varphi_1})+\|u_0-u_1\|_{\infty}a_n(C_{\varphi_1}).
			\end{equation*} 
			By interchanging the terms added/subtracted from the first integral we get that 
			\begin{equation*}
				a_n(M_{u_0}C_{\varphi_0}-M_{u_1}C_{\varphi_1})\leq \|u_1\|_{\infty} a_n(C_{\varphi_0}-C_{\varphi_1})+\|u_0-u_1\|_{\infty}a_n(C_{\varphi_0})
			\end{equation*} and the result follows.
		\end{proof}

		Let us now consider the approximation numbers of the differences of composition operators with the \emph{triangularly separated} symbols, i.e.\ symbols of the form
		$$
		\Phi(z_1,z_2)=(\phi(z_1),u(z_1)z_2)
		$$
		where $u, \phi \colon \mathbb{D}\rightarrow\mathbb{D}$ are self-maps of $\mathbb{D}$. A holomorphic function $f(z_1,z_2)=\sum_{j,k\geq 0}c_{j,k}z_{1}^{j}z_{2}^{k}\in H^2(\mathbb{D}^2)$ can be written as $f(z_1,z_2)=(\sum_{k\geq0}f_k(z_1))z_{2}^{k}$, where $f_k(z_1)=\sum_{j\geq0}c_{j,k}z_{1}^{j}$. Hence the map $J \colon H^2(\mathbb{D}^2)\rightarrow \bigoplus_{k=0}^{\infty}H^2(\mathbb{D})$, given by $Jf=(f_k)_{k\geq 0}$, is an isometric isomorphism (for details of this correspondence cf.\ \cite[Section 5.1]{LQRP19}). 
		
		\begin{thm}
			Let $u_0,u_1\in H^\infty(\mathbb{D})$ be weights satisfying $\|u_0\|\leq 1$, $\|u_1\|\leq 1$ and $\varphi_0$ and $\varphi_1$ be self-maps of the unit disc $\mathbb{D}$. Consider the symbols
				\begin{equation*}
									\Phi_0(z_1,z_2)=(\varphi_0(z_1),u_0(z_1)z_2), \quad \Phi_1(z_1,z_2)=(\varphi_1(z_1),u_1(z_1)z_2).
				\end{equation*}			
				For  positive integers $n_0,\dots,n_K$ and $N=n_0+\dots+n_K-K$,  the approximation numbers of the difference satisfy 
				\begin{align*}			
				a_N(C_{\Phi_{0}}-C_{\Phi_{1}})&	\leq \max\{\max_{0\leq k\leq K}(\min  [\|u_0\|_{\infty}^{k} a_{n_k}(C_{\varphi_0}-C_{\varphi_1})+\|u_{0}^{k}-u_{1}^{k}\|_{\infty}a_{n_k}(C_{\varphi_1}), \\
				&\|u_1\|_{\infty}^{k} a_{n_k}(C_{\varphi_0}-C_{\varphi_1})+\|u_{0}^{k}-u_{1}^{k}\|_{\infty}a_{n_k}(C_{\varphi_0})]),  \sup_{k>K}\|M_{u_{0}^{k}}C_{\varphi_0}-M_{u_{1}^{k}}C_{\varphi_1}\|\}.
			\end{align*}
		\end{thm}
		\begin{proof}
			For the given symbols $\Phi_{0}$ and $\Phi_{1}$ we have that
			$$
			J(C_{\Phi_{0}}-C_{\Phi_{1}})J^{-1}=\bigoplus_{k\geq 0}(M_{u_{0}^{k}}C_{\varphi_0}-M_{u_{1}^{k}}C_{\varphi_1}  )
			$$	
			and thus $a_n(C_{\Phi_{0}}-C_{\Phi_{1}})=a_n(T)$, where $T=\bigoplus_{k\geq 0} T_k$ and $T_k=M_{u_{0}^{k}}C_{\varphi_0}-M_{u_{1}^{k}}C_{\varphi_1}$.  By combining the previous proposition with \cite[Lemma 5.3]{LQRP19} we obtain the result.         
		\end{proof} 
		
		\begin{example}
			Let $u_0$ and $u_1$ be holomorphic weights such that $\|u_0\|_\infty,\|u_1\|_\infty<c<1$. Define the symbols
		\begin{align*}
				\Phi_{0}(z_1,z_2)&=(\varphi(z_1),u_0(z_1)z_2)\\
				\Phi_{1}(z_1,z_2)&=(\psi(z_1),u_1(z_1)z_2)			
		\end{align*}
			where $\varphi(z)$ and $\psi(z)$ are the symbols as given in Example \ref{eg:glued}.
			
			We first let $T_k=M_{u_{0}^{k}}C_{\varphi}-M_{u_{1}^{k}}C_{\psi}$, then $\|T_k\|\lesssim c^k$ and thus $\sup_{k>K}\|T_k\|\lesssim c^K$. Furthermore, by \cite[Theorem 1.2]{QS} for some $a_1,a_2>0$ we have that $a_n(C_\varphi)\lesssim e^{-a_1\sqrt{n}}$ and $a_n(C_\psi)\lesssim e^{-a_2\sqrt{n}}$. Hence taking $n_0=n_1=\dots=n_K=2^K$, we get that
		\begin{align*}
				\max_{0\leq k\leq K}\left( \right. \min &[\|u_0\|_{\infty}^{k} a_{n_k}(C_{\varphi}-C_{\psi_1})+\|u_{0}^{k}-u_{1}^{k}\|_{\infty}a_{n_k}(C_{\psi}),  \\ 
				&\|u_1\|_{\infty}^{k} a_{n_k}(C_{\varphi}-C_{\psi})+\|u_{0}^{k}-u_{1}^{k}\|_{\infty}a_{n_k}(C_{\varphi})] \left. \right)
			\lesssim e^{-\alpha 2^{K/2}}
		\end{align*}
		for some $\alpha>0$. Then since $n_0+n_1+\dots+n_K=( K+1)2^K$ we have
		$$
		a_{(K+1)2^K}(C_{\Phi_{0}}-C_{\Phi_{1}})\lesssim e^{-\alpha 2^{K/2}}
		$$
		from which we deduce  that
		$$
			a_{N}(C_{\Phi_{0}}-C_{\Phi_{1}})\lesssim e^{-\alpha_0\sqrt{N/\log N}}
		$$
		for some $\alpha_0>0$.		$\hfill\square$
		\end{example}

			\addtocontents{toc}{\SkipTocEntry}
		\subsection*{Acknowledgements} The authors would like to thank Herv\'{e} Queff\'{e}lec for several valuable and informative discussions. This study was conducted during the third named author's research year at Laboratoire de Math\'{e}matiques Blaise Pascal of Universit\'{e} Clermont Auvergne and she is deeply grateful for their hospitality.

		\bibliographystyle{abbrv}
		\bibliography{./referencesApproxNos}
		
	\end{document}